\documentclass{amsart}
\usepackage{amsmath, enumerate}
\usepackage{algorithm}
\usepackage{algpseudocode}
\usepackage{amsfonts}
\usepackage{amssymb}
\usepackage{amsthm}
\usepackage{color}
\usepackage{caption}
\usepackage{graphicx}
\usepackage{url}

\newcommand{\gam}{\Gamma}

\newcommand{\Q}{\mathbb{Q}}

\newcommand{\R}{\mathbb{R}}

\newcommand{\C}{\mathbb{C}}
\newcommand{\G}{\mathcal{G}}
\newcommand{\Go}{\mathcal{G}^{\mathrm{opp}}}

\newcommand{\bp}{\begin{problem}}

\newcommand{\ep}{\end{problem}}

\newcommand{\ba}{\begin{answer}}

\newcommand{\ea}{\end{answer}}

\newcommand{\ben}{\renewcommand{\theenumi}{\alph{enumi}}

\renewcommand{\labelenumi}{(\theenumi)}\begin{enumerate}}

\newcommand{\een}{\end{enumerate}}

\newcommand{\GL}{\mathrm{GL}}

\newcommand{\sgn}{\mathrm{sgn}}
\newcommand{\Id}{\mathrm{Id}}

\newtheorem{defin}{Definition}[section]
\newtheorem{thm}[defin]{Theorem}
\newtheorem{cor}[defin]{Corollary}
\newtheorem{rem}[defin]{Remark}
\newtheorem{lem}[defin]{Lemma}
\newtheorem{prop}[defin]{Proposition}
\newtheorem{quest}[defin]{Question}
\newtheorem{prob}[defin]{Problem}
\newtheorem{ex}[defin]{Example}

\DeclareCaptionType{result}[Table][List of Results]

\title[Graph Hamiltonicity and RAAGs]{Hamiltonicity via cohomology of right-angled Artin groups}
\begin{document}
\date{\today}

\author[R. Flores]{Ram\'{o}n Flores}
\address{Ram\'{o}n Flores, Department of Geometry and Topology, University of Seville, Spain}
\email{ramonjflores@us.es}

\author[D. Kahrobaei]{Delaram Kahrobaei}
\address{Delaram Kahrobaei, Department of Computer Science, University of York, UK,
New York University, Tandon School of Engineering, PhD Program in Computer Science, CUNY Graduate Center}
\email{dk2572@nyu.edu, delaram.kahrobaei@york.ac.uk}

\author[T. Koberda]{Thomas Koberda}
\address{Thomas Koberda, Department of Mathematics, University of Virginia, Charlottesville, VA 22904}
\email{thomas.koberda@gmail.com}

\begin{abstract}
Let $\Gamma$ be a finite graph and let $A(\Gamma)$ be the corresponding right-angled Artin group. We characterize
the Hamiltonicity of $\Gamma$ via the structure of the cohomology algebra of $A(\Gamma)$. In doing so, we define and develop
a new canonical graph associated to a matrix, which as a consequence provides a novel perspective on the matrix determinant.

\end{abstract}

\maketitle
%\tableofcontentso

\section{Introduction}
\label{Intro}
Let $\gam$ be a finite, undirected graph. A basic problem in graph theory that has been studied in a dizzyingly vast body of literature
(see~\cite{diestel-book} for instance)
 is determining when $\gam$ admits a ~\emph{Hamiltonian path} or a ~\emph{Hamiltonian cycle}, which is to say either a path
 or a cycle that visits every vertex of $\gam$ exactly once.
 One of the many reasons that the problem of determining whether $\gam$ is Hamiltonian (i.e.~admits such
 a path or cycle) is interesting is because it defies a straightforward characterization, and
 because it gives rise to the prototypical NP--complete
 decision problem~\cite{AB09,Minsky67,GJ1979}. In this paper, we develop an algebraic framework for characterizing Hamiltonicity of
 simplicial graphs via right-angled groups, and in particular the cohomology of these groups.

%\subsection{Right-angled Artin groups and Hamiltonian vector spaces}
To set up the discussion, let $\gam$ be a finite simplicial graph (i.e.~a graph that is also a simplicial complex)
with vertex set $V(\gam)$ and edge set $E(\gam)$. We retain the terminology \emph{simplicial} because it is standard in the theory
of right-angled Artin groups. Sometimes such graphs are called \emph{simple} graphs. We write
\[A(\gam)=\langle V(\gam)\mid [v,w]=1,\, \{v,w\}\in E(\gam)\rangle\] for the right-angled Artin group on $\gam$.

The isomorphism type of $A(\gam)$ determines the isomorphism type of $\gam$,
as is well-known from the work of many authors~\cite{Droms87,Sabalka09,KoberdaGAFA,koberda21survey}.
Thus, the combinatorics of $\gam$ should be reflected in the algebra of $A(\gam)$.

\subsection{Combinatorics versus group theory}

It is frequently useful to consider an object or a property in a certain context, and to identify the correct analogue in a different category.
The advantage of this approach is that there can be tools that may be available in the latter category but
not in the former, and these can be used to attack problems that appear unsolvable from the original point of view.
There are many examples where this strategy has been crucial in obtaining important results.

The present work fits in a series of papers (see references after Problem \ref{prob:property} below) where concepts and properties from
graph theory are translated into group theory and homotopy theory. On the one hand, we have the graph $\gam$, which is a discrete,
finitary object, amenable to investigation through combinatorics. On the other hand, we have the group $A(\gam)$,
which can be investigated not only through group theory, but also through the machinery of algebraic topology.
More precisely, graph-theoretic properties are reflected in the cohomology of the group with coefficients in a field, a
homotopical object which possesses rich structure (e.g.~the structure of an algebra, an action of the Steenrod algebra
when the field is finite, differential models if the field is $\Q$, etc.)~which in this way becomes suitable for applications to
 graph theory.
In particular, translation between combinatorics and group theory/topology provides a strategy for attacking the following problem,
whose resolution occupies the bulk of this note:

\begin{prob}\label{prob:hamilton}
Translate the property of Hamiltonicity from the category of graphs to the category of right-angled Artin groups.
\end{prob}

To proceed with resolving problem~\ref{prob:hamilton}, let $V$ and $W$ be vector spaces over a field $F$ with $V$ finite dimensional,
and let $q\colon V\times V\to W$ be a (anti-)symmetric
bilinear pairing on $V$.
We say that $(V,W,q)$ is~\emph{Hamiltonian} if whenever $\{w_1,\ldots,w_n\}$ is a basis for $V$ then there is a permutation
$\sigma\in S_n$ such that for all $1\leq i<n$, we have $q(w_{\sigma(i)},w_{\sigma(i+1)})\neq 0$.

The following two results are complete answers to Problem~\ref{prob:hamilton}.
The first intrinsically characterizes graphs admitting Hamiltonian paths
 via the cohomology ring of the corresponding right-angled Artin groups.

\begin{thm}\label{thm:main}
Let $\gam$ be a finite simplicial graph and let $F$ be a field. We set
\[V=H^1(A(\gam),F),\quad W=H^2(A(\gam),F),\] and let $q$ be the
cup product pairing $V\times V\to W$. Then $\gam$ admits a Hamiltonian path if and only if $(V,W,q)$ is Hamiltonian.
\end{thm}

We are able to prove an analogue of Theorem~\ref{thm:main} for graphs admitting Hamiltonian cycles as well. We will say
that $(V,W,q)$ above is~\emph{cyclic Hamiltonian} if for every basis $\{w_1,\ldots,w_n\}$ of $V$,
there is a permutation
$\sigma\in S_n$ such that for all $1\leq i<n$, we have $q(w_{\sigma(i)},w_{\sigma(i+1)})\neq 0$, and also
$q(w_{\sigma(n)},w_{\sigma(1)})\neq 0$.

\begin{thm}\label{thm:main-cycle-intro}
Let $\{\gam,F,V,W,q\}$ be as in Theorem~\ref{thm:main}.
Then $\gam$ admits a Hamiltonian cycle if and only if $(V,W,q)$ is cyclic Hamiltonian.
\end{thm}

The reader will note that
the property of the cohomology algebra of $A(\gam)$ being Hamiltonian is truly intrinsic to $A(\gam)$. The cohomology algebra of
$A(\gam)$ is canonically associated to the group. Moreover, since Hamiltonian vector spaces are defined without reference to a particular
basis, Hamiltonicity of a vector space is truly intrinsic to $(V,W,q)$.

Throughout the paper, we will concentrate on the proof of Theorem~\ref{thm:main}. Relatively mild generalizations are needed
to obtain Theorem~\ref{thm:main-cycle-intro}, and these are spelled out in Section~\ref{sec:cycle}.

Theorems~\ref{thm:main} and~\ref{thm:main-cycle-intro} fit into a broader circle of ideas.
We start with the following general guiding problem, which generalizes
Problem~\ref{prob:hamilton}:
\begin{prob}\label{prob:property}
Let $P$ be a property of finite simplicial graphs. Find a property $Q$ of groups such that $\gam$ has $P$ if and only if $A(\gam)$ has
$Q$.
\end{prob}
We insist that $Q$ be a property of the isomorphism type of the group, so that
in particular, $Q$ should be independent of any generating set.
Some satisfactory answers to Problem~\ref{prob:property} are known, for instance
when $P$ is the property of being a nontrivial join~\cite{Servatius1989},
being disconnected~\cite{BradyMeier01}, containing a square~\cite{Kambites09,KK2013gt}, being a co-graph~\cite{KK2013gt,KK2018jt},
being a finite tree or complete bipartite graph~\cite{HS2007}, admitting a nontrivial automorphism~\cite{FKK2019}, being $k$--colorable~
\cite{FKK2020a},
and fitting in a sequence of expanders~\cite{FKK2020exp}. Thus, Theorems~\ref{thm:main} and~\ref{thm:main-cycle-intro} resolve
Problem~\ref{prob:property} when $P$ is graph Hamiltonicity.

\subsection{Linear algebra}

The main innovation needed to prove
Theorems~\ref{thm:main} and~\ref{thm:main-cycle-intro} is a certain combinatorial object called the \emph{two--row graph} associated
to a matrix (see Subsection~\ref{ss:reduction} below). The combinatorics of this graph appear to be difficult to understand in general. The
crucial observation for us is that if a matrix is invertible, then its two--row graph contains a Hamiltonian path (Lemma \ref{lem:two-row}).
This result is in fact equivalent to the following purely linear algebraic fact.
Here, we say that two rows of a square matrix $A$ are \emph{null-connected}
if all the consecutive $2\times 2$ minors that they determine
(i.e.~minors in which the columns are consecutive in $A$) are singular; cf.~Definition \ref{def:null-conn} and Lemma~\ref{lem:two-row}.

\begin{prop}
\label{prop:alternative}

Let $A\in\GL_n(F)$. Then there exists a reordering of the rows of the matrix $A$ such that no consecutive rows are null-connected.
In particular, the two--row graph of $A$ always contains a Hamiltonian path.

\end{prop}

From a linear algebraic perspective, Proposition~\ref{prop:alternative} is the main result of this article, and
in addition to its major role in the proof of Theorem \ref{thm:main} and Theorem \ref{thm:main-cycle}, it is of independent interest. Establishing this fact is highly nontrivial, and requires the development of some linear algebraic tools which, to the knowledge of
the authors, are entirely novel. These ideas,
which can be found in Subsection~\ref{ss:square-trace}, give a perspective on the determinant of a matrix which appears not to have been
known before.

While we are partially inspired by
classical approaches of expanding the determinant which rely on the computation of $2\times 2$ minors such as Laplace
expansion, the Dodgson condensation formula,
and the Sylvester formula~\cite{MR3208413}, our point of view is radically different.

The two--row graph can be defined for an arbitrary matrix (which need not even be square).
Sometimes this graph is very sparse (and may in fact be disconnected), and
sometimes it may have many edges. Its general behavior appears to be quite complicated. The consequences of our methods
can be summarized as follows:
the two--row graph is a new invariant of matrices, and among square matrices, the invertible ones
have the property that the two--row graph is Hamiltonian.

\iffalse
The last section is devoted to the analysis of applications of our results. First, we discuss different implications of the previous theory for
Hamiltonicity testing, a crucial problem in complexity theory. For these applications, we generally
assume that all coefficients lie in the field with two elements, and thus all linear algebraic considerations become finitistic.
 Then, we develop a \emph{zero--knowledge proof protocol}, which turns out to be a variation of the standard one that uses Hamiltonian
 cycles. Recall that a zero--knowledge proof protocol (ZKP) is a probabilistic interactive proof system
   in which a prover demonstrates to the verifier that they are in possession
 of a piece of information
 without conveying additional data about the information.
The reader may consult the recent book of N.~Smart \cite{smart-book},
wherein there is an extensive cryptographic formulation of ZKPs (cf.~\cite{rosen-book}).
Smart frames the discussion of ZKPs in the setting of the graph isomorphism problem, which was at one point conjectured to be
NP--complete but which is now known to be solvable in quasi--polynomial time by a celebrated result of Babai~\cite{babai-proceedings}.
It is well-known, however, that ZKPs can be formulated in the context of any NP--complete problem~\cite{Blum87}.
We conclude the paper by describing some interplay between our algebraic notion of Hamiltonicity and the theory
 of random graphs, and by posing some problems that arise naturally in these contexts.
 \fi

\section{Hamiltonian vector spaces and the two--row graph}

In order to prove Theorem~\ref{thm:main}, we begin by gathering some preliminary facts.

\subsection{Cohomology of right-angled Artin groups}

In this subsection, we recall some basic facts about the structure of the cohomology algebra of a right-angled Artin group $A(\gam)$. The
result recorded here is easy and well--known, and follows from
standard methods in geometric topology together with the fact that the Salvetti complex associated to $\gam$ is a classifying space for
$A(\gam)$. Some details are spelled out in~\cite{FKK2020exp,koberda21survey}, for instance.

Let $V(\gam)=\{v_1,\ldots,v_n\}$ and $E(\gam)=\{e_1,\ldots,e_m\}$ be the vertices and edges of $\gam$,
and write $\smile$ for the cup product pairing
on $H^*(A(\gam),F)$, where here $F$ denotes an arbitrary field.

\begin{lem}\label{lem:coho}
Let $\gam$ be a finite simplicial graph. Then there are bases $\{v_1^*,\ldots,v_n^*\}$ for $H^1(A(\gam),F)$ and $\{e_1^*,\ldots,e_m^*\}$ for
$H^2(A(\gam),F)$ such that:
\begin{enumerate}
\item
We have $v_i^*\smile v_j^*= 0$ if and only if $\{v_i,v_j\}\notin E(\gam)$;
\item
We have $v_i^*\smile v_j^*=\pm e_k^*$ whenever $\{v_i,v_j\}=e_k\in E(\gam)$.
\end{enumerate}
\end{lem}

\subsection{The easy direction}
One direction of Theorem~\ref{thm:main} is straightforward.

\begin{lem}\label{lem:onlyif}
Let $\gam$ be a finite simplicial graph,
let $V=H^1(A(\gam),F)$, let $W=H^2(A(\gam),F)$, and let $q$ be the
cup product pairing $V\times V\to W$. If $(V,W,q)$ is Hamiltonian then $\gam$ admits a Hamiltonian path.
\end{lem}

\begin{proof}
Let $\{v_1,\ldots,v_n\}$ be a
list of the vertices in $\gam$, and write $\{v_1^*,\ldots,v_n^*\}$ for the corresponding dual cohomology
classes in $H^1(A(\gam),F)$ as in Lemma~\ref{lem:coho}.
Since $V$ is Hamiltonian, we may re-index $\{v_1^*,\ldots,v_n^*\}$ so as to assume that
$q(v_i^*,v_{i+1}^*)\neq 0$ for $1\leq i<n$. But then $\{v_i,v_{i+1}\}\in E(\gam)$ and hence $\{v_1,\ldots,v_n\}$ forms a Hamiltonian
path in $\gam$.
\end{proof}

The proof of the converse of Lemma~\ref{lem:onlyif} will occupy the remainder of this section.

\subsection{A linear algebraic reduction}\label{ss:reduction}
Here and for the remainder of the paper, we will use the notation $a_i^j$ for entries in a matrix, where the subscript denotes the
row and the superscript denotes the column. Rows of a matrix $A$ will be denotes $\{r_1,\ldots,r_n\}$, and the entry in the $j^{th}$ column
of $r_i$ will be denoted $r_i^j$.

Let $\{w_1,\ldots,w_n\}$ be an arbitrary basis for $V=H^1(A(\gam),F)$. We wish to show that there is a
re-indexing the basis $\{w_1,\ldots,w_n\}$ such that
after relabeling the basis vectors, we have $q(w_i,w_{i+1})\neq 0$ for $1\leq i<n$. We write
\[w_i=\sum_{j=1}^n a_i^j v_j^*\] for suitable coefficients $\{a_i^j\}_{1\leq i,j\leq n}\subset F$,
so that $A=(a_i^j)_{1\leq i,j\leq n}$ is an invertible matrix over $F$. Thus, we may view the rows of $A$ as corresponding to the expression
of $\{w_1,\ldots,w_n\}$ in terms of the basis $\{v_1^*,\ldots,v_n^*\}$, and a re-indexing of $\{w_1,\ldots,w_n\}$ is merely a permutation of the
rows of $A$.

A matrix $A\in \mathrm{M}_n(F)$ will be called~\emph{square--traceable} if for all $1\leq i<n$ there exists $1\leq j<n$ such that
the determinant of the minor \[A_{i}^j=\begin{pmatrix}a_i^j&a_{i}^{j+1}\\a_{i+1}^j&a_{i+1}^{j+1}\end{pmatrix}\] is nonzero.

The main technical fact we establish in this paper is the following purely linear algebraic statement:

\begin{lem}\label{lem:square-trace}
Let $A\in\GL_n(F)$. Then there is a permutation matrix $\sigma\in\GL_n(F)$ such that $\sigma A$ is square--traceable.
That is, $A$ is square--traceable, possibly after a permutation of the rows.
\end{lem}

If $A$ satisfies the conclusion of Lemma~\ref{lem:square-trace}, then we shall say that $A$ is \emph{permutation--square--traceable}.
Assuming Lemma~\ref{lem:square-trace}, we can complete the proof of Theorem~\ref{thm:main}, as follows from the next lemma.

\begin{lem}\label{lem:if}
Suppose $\gam$ is a finite simplicial graph that admits a Hamiltonian path,
let $V=H^1(A(\gam),F)$, let $W=H^2(A(\gam),F)$, and let $q$ be the
cup product pairing $V\times V\to W$. Then $(V,W,q)$ is Hamiltonian.
\end{lem}
\begin{proof}
Let $\{w_1,\ldots,w_n\}$ be a basis for $H^1(A(\gam),F)$, and let $\{v_1^*,\ldots,v_n^*\}$ be the standard basis for $H^1(A(\gam),F)$
arising from the duals of the vertices of $\gam$ as in Lemma~\ref{lem:coho}.
Assume that $q(v_i^*,v_{i+1}^*)\neq 0$ for $1\leq i<n$, which is possible since
$\gam$ contains a Hamiltonian path.
Writing \[w_i=\sum_{j=1}^n a_i^j v_j^*,\] we re-index the basis $\{w_1,\ldots,w_n\}$ so that the matrix $A=(a_i^j)_{1\leq i,j\leq n}$ is
square--traceable, which is possible by Lemma~\ref{lem:square-trace}.

A straightforward calculation shows that in expressing $q(w_i,w_{i+1})$ with respect to the vectors \[\{q(v_j^*,v_k^*)\mid 1\leq j< k\leq n\}\]
(which span $H^2(A(\gam),F)$), we have that the coefficient of the vector $q(v_j^*,v_{j+1}^*)$ is exactly $\det A_{i}^j$. Since
$\det A_{i}^j\neq 0$ for some choice of $j$ and since $q(v_j^*,v_{j+1}^*)\neq 0$
for all $j$ by assumption, we have $q(w_i,w_{i+1})\neq 0$. Thus, $(V,W,q)$ is Hamiltonian.
\end{proof}

\subsection{The two--row graph and square traceability}\label{ss:square-trace}

The permutation--square--traceability of a matrix $A$ can be formulated in terms of the combinatorics of a certain graph constructed from $A$.
Let $A\in\mathrm{M}_n(F)$. We view the rows of $A$ as $n$ linearly independent vectors over $F$, which we label $\{r_1,\ldots,r_n\}$.
We record these vectors as $r_i=(r_i^1,\ldots,r_i^n)$ for $1\leq i\leq n$. We construct
an undirected
 graph $\G(A)$, called the \emph{$2$--row graph} of $A$, as follows. The vertices of $\G(A)$ are simply the rows $\{r_1,\ldots,r_n\}$ of $A$.
Rows $r_i$ and $r_j$ of $A$ are connected by an edge in $\G(A)$ if and only if for some $1\leq k<n$, the minor
\[A_{i,j}^k=\begin{pmatrix}r_i^k&r_i^{k+1}\\r_j^k&r_j^{k+1}\end{pmatrix}\] is invertible, which the reader may compare with the definition
of square--traceability above. In other words, the rows $r_i$ and $r_j$ span an edge whenever
\[(\Lambda^2 A)(e_k\wedge e_{k+1})\] has a nonzero coefficient for $e_i\wedge e_j$, where here $\Lambda^2 A$ denotes the alternating
square of $A$ and where $\{e_1,\ldots,e_n\}$ denote standard basis vectors with respect to which $A$ is expressed.

One may check that for the identity matrix $\Id$, we have $\G(\Id)$ is a path, and for a dense open set of matrices in $\GL_n(\C)$ or
$\GL_n(\R)$, the $2$--row graph is complete. To see this last claim, we may restrict to the
dense and open subset of $n\times n$ matrices where all
entries are nonzero, and in the complement of the (closed subset with empty interior) where the determinant vanishes. For an
invertible matrix $A$ with all nonzero entries, two rows of $A$ fail to be connected by an edge in $\G(A)$ if and only if they are scalar
multiples of each other. it is then immediate that any two rows of $A$ are connected in $\G(A)$.

The graph $\G(A)$ will usually be considered for invertible matrices, though certainly one may consider non-invertible
and even non-square matrices, for which the definition of $\G(A)$ clearly still makes sense.

The usefulness of the $2$--row graph comes from the following trivial observation:
\begin{lem}
If $\G(A)$ contains a Hamiltonian path then $A$ is permutation--square--traceable.
\end{lem}
\begin{proof}
Choose a Hamiltonian path in $\G(A)$. Relabelling the vertices in this path
by the order in which the path visits them, we obtain a permutation matrix $\sigma$ such that $\sigma A$
is square--traceable.
\end{proof}

Thus, to prove Lemma~\ref{lem:square-trace}, it suffices to prove the following, which is the main innovation in this paper:

\begin{lem}\label{lem:two-row}
Let $A\in\GL_n(F)$. Then $\G(A)$ admits a Hamiltonian path.
\end{lem}

At first, it may seem difficult to prove anything general about the $2$--row graph of an invertible matrix. However, it is an enlightening exercise
to show that $\G(A)$ is always connected. The reader will note that connectedness of $\G(A)$ is obtained as an immediate consequence
of Lemma~\ref{lem:two-row}.

Observe that in general the converse of Lemma~\ref{lem:two-row} is false, by considering for example the matrix
\[A=\begin{pmatrix}0&1&0\\0&0&1\\0&1&0\end{pmatrix},\] whose two--row graph admits a Hamiltonian path but which is manifestly
non--invertible.

In order to establish Lemma~\ref{lem:two-row}, we will need to develop some more linear algebraic tools.
Henceforth, a \emph{minor} of a matrix $A$ is a square
 submatrix of $A$ obtained by deleting rows
and columns.

\begin{defin}\label{def:null-conn}
Let $r_i$ and $r_j$ be rows of $A$, and let $A_{i,j}$ be the $2\times n$ submatrix of $A$ spanned by these two rows.
We will say that rows $r_i$ and $r_j$ of $A$ are \emph{null--connected} if all $2\times 2$ minors of $A_{i,j}$ spanned by consecutive columns
are singular.
\end{defin}

We let $\Go(A)$ denote the graph spanned by the rows of $A$, with null--connectedness as the edge relation. Note that
$\G(A)$ and $\Go(A)$ are complementary graphs in the complete graph on the rows of $A$.

\begin{defin}\label{def:1-block}
Let $M$ be a submatrix of $A$. We will call $M$ a \emph{$1$--block} if the following conditions hold.
\begin{enumerate}
\item
All entries of $M$ are nonzero.
\item
$M$ has at least two rows and two columns, and all columns are consecutive in $A$.
\item
The subgraph of $\Go(A)$ spanned by the rows occurring in $M$ (viewed as rows of $A$) is connected.
\item
$M$ is maximal with respect to these conditions. That is, there is no submatrix $N$ of $A$ containing $M$ as a
proper submatrix, and which satisfies
the previous three conditions.
\end{enumerate}
\end{defin}

The following observation about $1$--blocks is elementary but useful, and justifies the terminology.

\begin{lem}\label{lem:row-space}
Let $M$ be a $1$--block in $A$. Then the row space of $M$ is one--dimensional.
\end{lem}
\begin{proof}
If $M$ has just two columns then the claim is clear. Let
$r_1=(a_1,\cdots, a_n)$ and $r_2=(b_1,\cdots, b_n)$ be two rows of $M$. Since these rows are null--connected in $A$, we have
that $(a_1,a_2)$ is proportional to $(b_1,b_2)$ by a nonzero constant, say $\lambda$. Similarly, $(a_2,a_3)$ and $(b_2,b_3)$ are proportional
by a nonzero constant, which must therefore be equal to $\lambda$. By induction on $n$, we have that $r_1$ and $r_2$ are proportional
by $\lambda$.
\end{proof}

Let $I\subset \{1,\ldots,n\}$ and let $1\leq s<t\leq n$. We will write $M^{s,t}_I$ for the submatrix
(not necessarily a $1$--block) spanned by rows in the index set $I$ and
columns with indices between $s$ and $t$ (inclusively). If $M_1=M_I^{s,t}$ and $M_2=M_{J}^{p,q}$ are such matrices, we will write
$M_1\triangleleft M_2$ if $p=t+1$.

The following lemma shows that the structure of $1$--blocks is highly constrained.

\begin{lem}\label{lem:1-block}
Let $M_1=M_I^{s,t}$ and $M_2=M_{J}^{p,q}$ be $1$--blocks. If $M_1\neq M_2$ then $M_1$ and $M_2$ are disjoint as submatrices of $A$.
\end{lem}
\begin{proof}
Without loss of generality, $s\leq p$. Suppose that the $(i,k)$--entry of $A$ belongs to both blocks.
This means that $i\in I\cap J$ and $p\leq k\leq t$.
We claim that $I=J$ in this case, which is easily seen to be sufficient to establish the lemma (cf.~Lemma~\ref{lem:row-space}).

Let $R\subseteq I\setminus \{i\}$ denote the set of indices such that each row indexed by
$R$ is null--connected to the row $r_i$. Since $M_1$ is a $1$--block,
we have that $R\neq \emptyset$. Let $\ell\in R$. Let us then consider the submatrix $N$ of $A$ spanned by
$r_i$ and $r_{\ell}$, which must have the following shape:
\[N=\begin{pmatrix}\cdots&a_i&b_i&c_i&d_i&\cdots\\ \cdots&a_{\ell}&b_{\ell}&c_{\ell}&d_{\ell}&\cdots\end{pmatrix},\]
where here the entries $c_i$ and $c_{\ell}$ lie in the column indexed by $p$.

Now, since $M_1$ and $M_2$ are $1$--blocks, we must have that $0\notin \{b_i,c_i,d_i, b_{\ell}\}$.
Since $r_i$ and $r_{\ell}$ are null--connected, this forces $c_{\ell}$ and $d_{\ell}$ to be nonzero as well, and so the vectors $(c_i,d_i)$ and
$(c_{\ell},d_{\ell})$ are nonzero multiples of each other. It is clear that this
forces all entries in $r_{\ell}$ to be nonzero for columns indexed between $p$
and $q$.

The vector $(r_i^k,r_i^{k+1})$ is proportional to the vector $(r_{\ell}^k,r_{\ell}^{k+1})$ for $p\leq k<q$
by the null--connectedness of $r_i$ and $r_{\ell}$ (where here $r_i^k$ denotes the $k^{th}$ column entry in the row $r_i$).
It follows that these pairs of vectors \[\{(r_i^k,r_i^{k+1}),(r_{\ell}^k,r_{\ell}^{k+1})\}_{p\leq k<q}\] are all proportional with the
same nonzero constant of proportionality (cf.~the proof of Lemma~\ref{lem:row-space}).
We then conclude that the $\ell$--row $(r_{\ell}^p,\ldots,r_{\ell}^q)$ of $M_2$ is proportional to the $i$--row
$(r_i^p,\ldots,r_i^q)$ of $M_2$ by a nonzero constant of proportionality, whence the maximality part of the definition of a
$1$--block implies that $\ell\in J$. Since the rows indexed by $I$ span a connected subgraph of
$\Go(A)$, an easy induction shows that $I\subseteq J$. Evidently, the reverse inclusion holds by the same argument.
\end{proof}

We remark that it is not difficult to modify this argument and show that if $M_I^{p,q}$ and $M_J^{q+1,s}$ are $1$--blocks then
$I\cap J=\emptyset$.
Lemma~\ref{lem:1-block} also has the following easy consequence, which follows from the maximality in the definition of a $1$--block:

\begin{cor}\label{cor:partition}
A matrix $A$ admits a unique partition into submatrices of the following three types:
\begin{enumerate}
\item
$1$--blocks;
\item
$1\times 1$ nonzero matrices which do not belong to any $1$--block;
\item
$1\times 1$ zero matrices.
\end{enumerate}
\end{cor}

\begin{ex}\label{ex:7x7}
To illustrate the foregoing concepts, consider the matrix
\[A=\begin{pmatrix}1&1&1&0&0&0&1\\0&1&0&1&0&0&1\\1&1&1&0&1&0&1\\
0&1&0&0&1&0&1\\1&1&1&0&0&1&0\\1&0&1&0&0&0&1\\1&0&0&0&1&0&1\end{pmatrix}.\]
Some tedious but straightforward calculations show that $A$ is invertible. The graph $\G(A)$ is highly connected; in fact, the only missing
edges are between rows $1$ and $3$, and between rows $6$ and $7$, and therefore these are the only pairs of rows
in $A$ that are null--connected.
Setting $I=\{1,3\}$, the submatrix $M_I^{1,3}$ forms the unique $1$--block in $A$.
\end{ex}

\begin{defin}
Let $M$ be a minor in $A$ with at least two rows.
We say that $M$ is a ~\emph{$1$--minor} if $M$ is contained in a $1$--block, and if the columns of $M$ are
consecutive in $A$.
\end{defin}

\begin{defin}
Let $A$ be an $n\times n$ matrix and let
\[M_1\triangleleft M_2 \triangleleft \cdots \triangleleft M_r\] be a sequence of minors in $A$
of the form $M_i=M_{J_i}^{s_i,t_i}$. We say that this
sequence is a \emph{$1$--track} if the following conditions hold:
\begin{enumerate}
\item
For each $i$, the matrix is a nonzero $1\times 1$ matrix or a $1$--minor.
\item
For $1\leq i<r$, it is not the case that $M_i$ and $M_{i+1}$ belong to a common $1$--block.
\end{enumerate}
We will say that a $1$--track $T$ is~\emph{complete} if the total number of columns in $T$ is $n$.
We say that two $1$--tracks $\{M_i\}_{1\leq i\leq r}$ and $\{N_j\}_{1\leq j\leq s}$ are \emph{different} if they are distinct sequences of matrices.
For $\sigma\in S_n$, we say that a string
\[\mathfrak{a}=\left(a_{\sigma(1)}^1,\ldots,a_{\sigma(n)}^n\right)\] of entries of $A$
\emph{belongs} to the (complete) $1$--track $T$ if every entry in $\mathfrak{a}$ belongs to
some minor in $T$.
\end{defin}

The following is a crucial fact about $1$--tracks.

\begin{lem}\label{lem:unique}
Let $\mathfrak{a}=(a_{\sigma(1)}^1,\ldots,a_{\sigma(n)}^n)$ be a string consisting of nonzero entries of $A$. Then $\mathfrak{a}$ belongs to
exactly one complete $1$--track.
\end{lem}
\begin{proof}
We will construct a $1$--track $T$ to which $\mathfrak{a}$ belongs, in a canonical way. If $a_{\sigma(1)}^1$ does not belong to a $1$--block,
then we set $M_1=a_{\sigma(1)}^1$. If for all $1\leq i\leq n$ we have $a_{\sigma(i)}^i$ belong to the same $1$--block then we may set
$A=M_1$. Otherwise, there is a $k$ so that $a_{\sigma(j)}^j$ belongs to a single $1$--block $B$ for $j\leq k$, but $a_{\sigma(k+1)}^{k+1}$
does not.
We then let $M_1$ be the $1$--minor spanned by rows indexed by $\{\sigma(1),\ldots,\sigma(k)\}$ and columns indexed by $1\leq j\leq k$.

To construct $M_2$, we restart the construction of the previous paragraph at $a_{\sigma(k+1)}^{k+1}$, and thus inductively construct the
$1$--track $T$ to which the string $\mathfrak{a}$ belongs. It is immediate that the resulting $1$--track is complete.
Note that this construction is canonical and hence the resulting $1$--track is unique (cf.~Remark~\ref{rem:unique} below).
\end{proof}

\begin{rem}\label{rem:unique}
In order to obtain uniqueness of the $1$--track in the proof of Lemma~\ref{lem:unique}, we are
using the uniqueness of $1$--blocks from Lemma~\ref{lem:1-block} in an essential way. If $M'$ were a $1$--minor
containing both $a_{\sigma(i)}^i$ and $a_{\sigma(i+1)}^{i+1}$ that meets the $1$--minor $M_k\in T$, then $M'$ is a submatrix of $M_k$.
This follows, as Lemma~\ref{lem:1-block} implies that $M_k$ and $M'$ lie in a single
$1$--block, so that $M'$ will be subsumed as a submatrix of $M_k$ in the construction of $T$.
\end{rem}

Let $T$ be a complete $1$--track in a matrix $A$, and let $\Sigma_T\subset S_n$ be the collection of permutations $\sigma\in S_n$ such that
$\mathfrak{a}_\sigma=(a_{\sigma(1)}^1,\ldots,a_{\sigma(n)}^n)$ belongs to $T$. The next lemma shows that every ``cluster" of nonzero entries
coming from a $1$--minor in a matrix $A$ contributes nothing to the determinant of $A$.

\begin{lem}\label{lem:perm-cancel}
Let $T=\{M_i\}_{1\leq i\leq r}$ be a complete $1$--track in $A$ that contains at least one $1$--minor. Then
\[\sum_{\sigma\in\Sigma_T}\sgn(\sigma)\prod_{i=1}^n a_{\sigma(i)}^i=0.\]
\end{lem}
\begin{proof}
Without loss of generality, we will assume that $M_1$ is a $1$--minor, which has $k>1$ rows. Since $M_1$ is contained in a $1$--block, its
rows are all proportional to each other. It follows by an easy calculation that
\[\prod_{i=1}^k a_{\sigma(i)}^i\] is constant for $\sigma\in \Sigma_T$.

We identify $S_k$ with the permutations of
the $k$ rows of $A$ corresponding to the rows of $M_1$, and which fix the remaining rows of $A$. Let
$\mathfrak{a}_\sigma=(a_{\sigma(1)}^1,\ldots,a_{\sigma(n)}^n)$ belong to $T$. Then for all $\tau\in S_k$, we have
\[\left(a_{\tau\sigma(1)}^1,\ldots,a_{\tau\sigma(k)}^k, a_{\sigma(k+1)}^{k+1},\ldots,a_{\sigma(n)}^n\right)\] also belongs to $T$. It is
not difficult to see then that
for such a fixed $\sigma$, we have \[\sum_{\tau\in S_k}\sgn(\tau)
\left(a_{\tau\sigma(1)}^1\cdots a_{\tau\sigma(k)}^k\cdot a_{\sigma(k+1)}^{k+1}\cdots a_{\sigma(n)}^n\right)=0.\] Indeed, this follows
simply from the fact that half of the permutations in $S_k$ have sign $1$ and half have sign $-1$, and the previous observation
that the product $a_{\tau\sigma(1)}^1\cdots a_{\tau\sigma(k)}^k$ is independent of $\tau\in S_k$.

Finally, consider the sum \[\sum_{\sigma\in\Sigma_T}\sgn(\sigma)\prod_{i=1}^n a_{\sigma(i)}^i.\]
The previous considerations show that the total contribution from
strings $\mathfrak{a}_{\sigma}$ whose tail is of the form $(a_{\sigma(k+1)}^{k+1},\ldots,a_{\sigma(n)}^n)$
is zero. It follows that the total sum vanishes,
as claimed.
\end{proof}

We can now finish the proof of the main result.

\begin{proof}[Proof of Lemma~\ref{lem:two-row}]
It suffices to show that if $A\in\mathrm{GL}_n(F)$ then there is a reordering of the rows of $A$ so that no pair of consecutive rows is
null--connected. So, suppose the contrary, and we shall show that $\det(A)=0$.

Let $\mathfrak{a}_\sigma=(a_{\sigma(1)}^1,\ldots,a_{\sigma(n)}^n)$ be a string consisting of nonzero entries of $A$. Observe that there
must be a $1$--minor containing two consecutive entries of $\mathfrak{a}_{\sigma}$. Indeed, otherwise there would be a reordering of the
rows of $A$ such that no pair of consecutive rows is null--connected.

Let $\Sigma_N\subset S_n$ denote the set of permutations $\sigma$ for which the string
\[\mathfrak{a}_\sigma=\left(a_{\sigma(1)}^1,\ldots,a_{\sigma(n)}^n\right)\] consists of nonzero entries of $A$. By definition, we have
\[\det(A)=\sum_{\sigma\in \Sigma_N}\sgn(\sigma)\prod_{i=1}^n a_{\sigma(i)}^i.\] By Lemma~\ref{lem:unique}, every such string
$\mathfrak{a}_{\sigma}$ belongs to a unique complete $1$--track, and conversely for each complete $1$--track $T$ and each
$\mathfrak{a}_{\sigma}$ belonging to $T$ we have that $\sigma\in\Sigma_N$.

It follows that there is a collection of distinct complete $1$--tracks $\{T_1,\ldots,T_m\}$ such that the collection of strings belonging to these
$1$--tracks partitions $\{\mathfrak{a}_{\sigma}\}_{\sigma\in\Sigma_N}$.
Lemma~\ref{lem:perm-cancel} implies that $\det(A)=0$, as claimed.
\end{proof}

\begin{rem}

As stated in the introduction (see Proposition \ref{prob:property}), Lemma \ref{lem:two-row} admits a purely algebraic formulation,
which might be of independent interest in linear algebra. The equivalence between the two statements is immediate.

\end{rem}

\section{From Hamiltonian paths to Hamiltonian cycles}\label{sec:cycle}

In this section, we extend the definition of the two--row graph $\G(A)$ to allow for cyclic permutations of the columns of $A$.
To that end, we define the \emph{cyclic two--row graph} $\G^c(A)$ as follows. The vertices of $\G^c(A)$ are the rows of $A$. Two rows
$r_i$ and $r_j$ are adjacent if they are adjacent in $\G(A)$, or if the minor
\[R_{i,j}^n=\begin{pmatrix}r_i^n&r_i^1\\r_j^n&r_j^1\end{pmatrix}\] is
invertible. Note that, for example, $\G^c(\Id)$ is a cycle.
We establish a variation on Lemma~\ref{lem:two-row} by proving the following result:

\begin{prop}\label{prop:two-row}
Let $A\in\mathrm{GL}_n(F)$. Then $\G^c(A)$ contains a Hamiltonian cycle.
\end{prop}

Recall that with the same notation as above, we consider a triple $(V,W,q)$ and say that it is \emph{cyclic Hamiltonian} if for every basis
$\{v_1,\ldots,v_n\}$ of $V$, there is a permutation $\sigma$ such that $q(v_{\sigma(i)},v_{\sigma(i+1)})\neq 0$ for $1\leq i\leq n$, and
where the indices are considered cyclically. The following is a restatement of Theorem~\ref{thm:main-cycle-intro}.

\begin{thm}\label{thm:main-cycle}
Let $\gam$ be a finite simplicial graph and let $F$ be a field. We set
\[V=H^1(A(\gam),F),\quad W=H^2(A(\gam),F),\] and let $q$ be the
cup product pairing $V\times V\to W$. Then $\gam$ admits a Hamiltonian cycle if and only if $(V,W,q)$ is cyclic Hamiltonian.
\end{thm}

As with Theorem~\ref{thm:main}, there is an easier and a harder direction. The proof of the easier direction is analogous to
the proof of Lemma~\ref{lem:onlyif} above, and the harder direction follows
from Proposition~\ref{prop:two-row} by a proof that is analogous to that of Theorem~\ref{thm:main}.

\subsection{Proving Proposition~\ref{prop:two-row}}

Many of the concepts from Subsection~\ref{ss:square-trace} generalize verbatim or nearly verbatim to the cyclic two--row graph, after some
obvious modifications. For one, null--connectedness is now a stronger condition, as it is the complement of the adjacency relation in $\G^c(A)$.
The definition of a \emph{cyclic $1$--block} is identical to the definition of a $1$--block in
Definition~\ref{def:1-block}, except that columns are arranged in a cyclic order and the relation of being consecutive is correspondingly
weakened. The superscripts in the notation $M_I^{s,t}$ are taken cyclically. So, if $A\in\mathrm{M}_6(F)$ for example and
if $I=\{1,2,5\}$,
then the submatrix $M_I^{4,1}$ is spanned by the rows indexed by $I$, and by columns $\{4,5,6,1\}$. The sixth and first columns
are thus viewed as consecutive.
For the purposes of our analyses, it is conceptually useful to imagine the matrix $A$ written on a torus, so that the top and bottom rows
are consecutive, and the leftmost and rightmost columns are consecutive.

With the modified notion of consecutiveness, the proofs of Lemma~\ref{lem:row-space} and
Lemma~\ref{lem:1-block} for cyclic $1$--blocks are the same. Similarly, the definition of a
\emph{cyclic $1$--minor} transfers to cyclic $1$--blocks easily.
Slight care should be taken when defining a \emph{cyclic $1$--track}: if $\{M_i\}_{1\leq i\leq r}$ is a cyclic $1$--track in $A$, then writing
$M_1=M_{J_1}^{s_1,t_1}$ and $M_r=M_{J_r}^{s_r,t_r}$, we require that $t_r\leq s_1$ in the cyclic order on the columns of $A$.
We obtain immediately the notion of a string $\mathfrak{a}$ belonging to a cyclic $1$--track, and now the string is considered up
to a cyclic permutation. It is straightforward to generalize Lemma~\ref{lem:unique} to cyclic $1$--tracks. Lemma~\ref{lem:perm-cancel}
then generalizes to cyclic $1$--tracks containing at least one cyclic $1$--minor, simply by cyclically permuting the columns of $A$ so that
a cyclic $1$--minor appears as the first minor in a $1$--track.

\begin{proof}[Proof of Proposition~\ref{prop:two-row}]
The proof is a reprise of the proof of Lemma~\ref{lem:two-row}: indeed, suppose the contrary.
Then, for any permutation $\sigma\in S_n$, at least one pair of rows $\{r_{\sigma(i)},r_{\sigma(i+1)}\}$ is null--connected,
for $1\leq i\leq n$. Let \[\mathfrak{a}_\sigma=(a_{\sigma(1)}^1,\ldots,a_{\sigma(n)}^n)\] be a string consisting of nonzero entries of $A$, and let
$T$ be the unique cyclic $1$--track to which $\mathfrak{a}_\sigma$ belongs. It must be the case that two consecutive entries in
$\mathfrak{a}_\sigma$ belong to a $1$--minor, since otherwise we would violate our initial null--connectedness assumption, and therefore we
 must have that $T$ contains at least one $1$--minor. The total sum of contributions to $\det(A)$ from $T$ is zero. It follows that $\det(A)=0$,
 a contradiction.
\end{proof}

\begin{rem}

Proposition \ref{prop:two-row} also admits an algebraic formulation analogous to Proposition \ref{prop:alternative}, by
 adopting the less restrictive notion of consecutiveness defined in this section.

\end{rem}

\begin{ex}
Let us briefly revisit Example~\ref{ex:7x7} above, in the context of the cyclic concepts.
If one allows for wraparounds in the columns, then there is one more $1$--block
in $A$. Indeed, set $J=\{6,7\}$. Then the submatrix $M_J^{7,1}$ is also a $1$--block. The remaining entries in $A$ are accounted for
according to Corollary~\ref{cor:partition}. The graphs $\G(A)$ and $\G^c(A)$ are isomorphic to each other.
\end{ex}

\section{Some questions}

The discussion in this paper naturally leads to some further questions. There are of course those that arise from the interest in
graph Hamiltonicity from the point of view of complexity theory. These questions are thoroughly discussed in the literature;
see for instance ~\cite{AB09,Minsky67,GJ1979,rosen-book}, and as such we will not comment on them in detail.

Because the main new ideas in this article come from linear algebra and in particular the two--row graph, we close with some
problems that arise from that perspective.

Let $q$ be a power of a prime and let $\mathbb{F}_q$ denote the field with $q$ elements.
Consider the uniform measure on the finite group $\mathrm{GL}_n(\mathbb{F}_q)$, thus producing a model for random matrices:

\begin{quest}
Let $A\in \mathrm{GL}_n(\mathbb{F}_q)$ be random. What is the probability that $\G(A)$ is complete, as a function of $n$? How
does this probability depend on $q$?
\end{quest}

A random matrix in $\mathrm{GL}_n(\mathbb{F}_q)$ will have the property that $\G(A)$ is Hamiltonian, and so we obtain a model
for a random Hamiltonian graph.

In the non-random context, the realization problem is interesting and is likely quite difficult for large $n$:

\begin{quest}\label{q:realization}
Let $\gam$ be a fixed graph Hamiltonian graph. Is $\gam\cong \G(A)$ for an invertible matrix $A$? What further restrictions
besides Hamiltonicity does the invertibility of $A$
impose?
\end{quest}

The realization problem is much easier if one does not require invertibility of $A$ and if one drops the requirement that $A$ be square.
For a non-square matrix $A$, we adopt the obvious generalization of $\G(A)$.

\begin{prop}
Let $\gam$ be an arbitrary finite simplicial graph. Then there is a matrix $A$ such that $\gam\cong\G(A)$.
\end{prop}
\begin{proof}[Proof sketch]
Let $\gam$ have $n$ vertices, labeled $\{v_1,\ldots,v_n\}$.
We build the matrix $A$, one vertex of $\gam$ at a time, padding with zeros appropriately.
Let $A_1$ be a $1\times 1$ matrix with a single $1$ entry. We set $A_2$ to be a $2\times 2$ matrix, which is the identity if $v_1$ and $v_2$
span an edge, and otherwise is the matrix \[\begin{pmatrix}1&0\\0&0\end{pmatrix}.\] Now suppose that $A_j$ is constructed. We construct
$A_{j+1}$ as follows.
\begin{enumerate}
\item
Add a row of zeros at the bottom of $A_j$ to get a matrix $B_{1,j}$, and set the counter $i=1$.
\item
Add a column of zeros on the right of $B_{1,j}$ to get a matrix $B_{2,j}$. If the counter $i$ is equal to $j+1$ then stop and set $A_{j+1}=
B_{2,j}$.
\item
If $v_i$ and $v_{j+1}$ span an edge in $\gam$, add two columns to $B_{2,j}$ which are zero everywhere except in rows $i$ and ${j+1}$,
wherein the entries are $(1,0)$ and $(0,1)$ respectively. Otherwise, do nothing to the matrix, increase the counter $i$ by one,
 and return to step $2$.
\item
Rename the resulting matrix $B_{1,j}$, increase the counter $i$ by one, and return to step $2$.
\end{enumerate}

We set $A=A_n$ with an extra column of zeros on the right.
It is straightforward to check that two rows of $A$ will be null--connected if and only if the corresponding vertices of $\gam$ do not
span an edge.
\end{proof}

One can also consider the problem of determining the permutation--square--traceability of a matrix $A$, or equivalently the Hamiltonicity of
the graph $\G(A)$.
Observe that to check by brute force that a matrix is permutation--square--traceable requires checking
all possible permutations of the rows and examining the edge relation between consecutive rows.

\begin{quest}\label{q:square-trace}
Let $A\in \mathrm{M}_n(\mathbb{F}_2)$. What is the complexity of the problem to decide whether or not $A$ is
permutation--square--traceable?
\end{quest}

Note that Question~\ref{q:square-trace} is only interesting for non-invertible matrices, by the main results of this paper.
Observe that a Hamiltonian path in $\G(A)$ is a certificate of the permutation--square--traceability of a matrix,
so that the problem in Question~\ref{q:square-trace} is
clearly in NP. We suspect that it should be NP--complete.

\section*{Acknowledgments}

\vspace{.5cm}

We thank F. Abeles, R. Brualdi, A. de Camargo, T. Markham and J.M. Pe\~{n}a for helpful comments, and the referee for his/her careful report.
Ram\'{o}n Flores is supported by FEDER-MEC grant MTM2016-76453-C2-1-P and FEDER grant US-1263032 from the Andalusian
Government. Thomas Koberda is partially supported  by an
Alfred P. Sloan Foundation Research Fellowship, by NSF Grant DMS-1711488, and by NSF Grant DMS-2002596.
Delaram Kahrobaei is supported in part by a
Canada's New Frontiers in Research Fund, under the Exploration grant entitled ``Algebraic Techniques for Quantum Security".
We thank the University of York for hospitality while part of this research
 was conducted.
\bibliographystyle{amsplain}
\bibliography{ref}

\end{document}